\documentclass{amsart}
\usepackage{color}
\usepackage{amsmath, amssymb, amsfonts, amstext, amsthm, textcomp, comment,bbm, mathtools}
\usepackage{nccmath,enumitem, braket}
\usepackage[left=3cm, right=3cm]{geometry}
\usepackage{lineno}

  \usepackage[normalem]{ulem}
\usepackage{chngcntr}
\usepackage{enumitem}
\usepackage{float}

\newtheorem{theorem}{Theorem}[section]
\newtheorem{lemma}[theorem]{Lemma}
\newtheorem{corollary}[theorem]{Corollary}
\newtheorem{proposition}[theorem]{Proposition}

\theoremstyle{definition}
\newtheorem{definition}[theorem]{Definition}
\newtheorem{example}[theorem]{Example}

\newtheorem{note}{Note}

\theoremstyle{remark}
\newtheorem{remark}[theorem]{Remark}

\numberwithin{equation}{section}

\newcommand{\N}{\mathbb{N}}
\newcommand{\Z}{\mathbb{Z}}

\newcommand{\R}{\mathbb{R}}

\newcommand{\calB}{\mathcal{B}}

\DeclareMathOperator{\Range}{R}

\DeclareMathOperator{\calD}{\mathcal{D}}
\usepackage[colorlinks=true]{hyperref}

\newcommand{\tr}{\operatorname{tr}} %

\begin{document}

\title{Positive Linear Maps on Second Symmetric Product Spaces}


            
 

\author{Pavankumar Raickwade}
\address{Department of Mathematics, IIT Madras, India}
\curraddr{}
\email{prraickwade@gmail.com}
\thanks{}

\author{K. C. Sivakumar}
\address{Department of Mathematics, IIT Madras, India}
\curraddr{}
\email{kcskumar@iitm.ac.in}
\thanks{}



\date{}

\dedicatory{}

\begin{abstract} 
   Let $X^{(2)}$ denote the second symmetric product space of a partially ordered vector space $X$, endowed with the projective cone. A characterization of linear maps $T\colon X^{(2)}\to X^{(2)}$ which preserve the set of all positive decomposable vectors, is proved. As applications of this result, an alternative proof, as well as an infinite dimensional generalization, of a representation theorem for (i) automorphisms on the completely positive cone and (ii) linear preservers of CP-rank-1 matrices, are presented. It is also shown that if $T$ preserves the set of all decomposable vectors, then so does the Drazin inverse, $T^D$ (if it exists). The case of the Moore-Penrose inverse is also investigated.
\end{abstract}

\maketitle

\section{Introduction}

For a vector space $X$, let $X\otimes X$ denote its {\it tensor product} with itself. The subspace $X^{(2)}\subseteq X\otimes X$, of all {\it symmetric} tensors of degree two, is called the {\it second symmetric} product space of $X$. For a nonzero element $A\in X^{(2)}$, the {\it rank} of $A$ is the least integer $k$ such that $A=\sum_{i=1}^k \lambda_i (x_i\otimes x_i)$ where $x_1,\ldots, x_k$ are linearly independent vectors in $X$ and $\lambda_1,\ldots, \lambda_k$ are nonzero scalars in $\R.$ The rank atmost one elements in $X^{(2)}$ are called the {\it decomposable elements}. For a linear map $f\colon X\to X$, the linear map $P_2(f)\colon X^{(2)}\to X^{(2)}$, given by 
\[
    P_2(f)(x\otimes y) = f(x)\otimes f(y), \quad \forall x,y\in X,
\]
is called the {\it second symmetric power} of $f$.
A linear map $T\colon  X^{(2)}\to X^{(2)}$ is called {\it rank one non-increasing} (or a linear preserver of decomposable elements) if, $\text{rank}(A)=1$ implies $\text{rank}(T(A))\leq 1$. In \cite[Theorem 1]{Lim01021990}, the author characterized linear maps $T\colon X^{(2)}\to X^{(2)}$ that are rank one non-increasing, as follows.

\begin{theorem}\label{thm::Lim}
    Let $T\colon X^{(2)}\to X^{(2)}$ be linear. Then $T$ is rank one non-increasing if and only if either of the following holds.
    \begin{enumerate}[label=\upshape(\roman*)]
        \item $T=cP_2(f)$ for some linear map $f\colon X\to X$ and $c\in \R\setminus\{0\}$,
        \item $T(\cdot)=\phi(\cdot) y\otimes y$ for some linear functional $\phi\colon X^{(2)}\to \R$ and $y\in X$.
    \end{enumerate}
\end{theorem}

Suppose $X$ is a partially ordered vector space with a cone $X_+$. It is then of interest to describe an {\it appropriate} cone in the tensor space $X\otimes X$. There are various ways to do this and this study was initiated in $1960$s by Merklin \cite{1964}, Peressini and Sherbert \cite{Peressini}, Haluniki and Phelps \cite{HULANICKI1968177}. Among others, the most intuitive and the most studied is the {\it projective cone}, given by 
\begin{equation}\label{eq::projective-cone}
    X_+\otimes^\pi X_+:=\Set{\sum_{i=1}^k x_i\otimes y_i\ |\ k\in \N, ~x_i, y_i\in X_+}.
\end{equation}
The study of projective cone has received a lot of attention recently, see, for e.g., \cite{Grobler_Labuschagne_1988, tensor-anke, Wortel2019}. The restriction of projective cone on $X^{(2)}$ is denoted by $X^{(2)}_+$, i.e., $X^{(2)}_+:=(X_+\otimes^\pi X_+)\cap X^{(2)}$. We are interested in obtaining an analogue of Theorem~\ref{thm::Lim}, for the set of all {\it positive} decomposable vectors. We also investigate the question of whether some generalized inverses of rank one non-increasing maps, also share such a property. 

Another cone of interest, particularly in matrix theory and semidefinite optimization, is the {{\it completely positive cone},} described as
\[
    CP(X_+):=\Set{\sum_{i=1}^n u_i\otimes u_i\ |\ u_i\in X_+,~ n\in \N}\subseteq X^{(2)}_+.
\]
For a finite dimensional space $X$ with a closed and generating cone $X_+$, a representation theorem for $Aut(CP(X_+))$, the set of all cone automorphisms, was obtained in \cite{GOWDA20133862}. We are interested in obtaining a representation theorem for $Aut(CP(X_+))$ for an ordered normed space $X$ with a closed and generating cone $X_+$.

This paper is organized as follows. We collect all the necessary notions and terminology, in Section~\ref{prelims}. In Section~\ref{sec::positive-decompo-preservers},  we characterize all linear maps on $X^{(2)}$ that preserve the set of all positive decomposable elements (Theorem~\ref{thm::Lim_2.0}). Using this result, in Section~\ref{sec::CP-rank}, we obtain a representation of linear maps that preserve elements of CP-rank $1$, extending one of the main results in \cite{Beasley2019-sp} to infinite dimensional spaces, and in Section~\ref{Sec::Aut(CP)}, we obtain a characterization of $Aut(CP(X_+))$ (Theorem~\ref{thm::Aut-CP}), generalizing the main result of  \cite{GOWDA20133862}. In Section~\ref{sec::gen-inve}, we consider generalized inverses of rank one non-increasing maps, where we show that the Drazin inverse shares this property (Theorem~\ref{thm::inverse-of-rank-non-increasing}), while the Moore-Penrose inverse does not. We also obtain a necessary and sufficient condition for the Moore-Penrose inverse to possess this property (Theorem~\ref{thm::rank-non-increasing-Moore-Penrose-inverse}).

\section{Preliminaries}\label{prelims}

\subsection{Partially ordered vector spaces}
Let $X$ be a real vector space and $X'$ denote the algebraic  dual of $X$, i.e. the space of all linear functionals on $X$. A partial order $\leq$ on $X$ is called a {\it vector space order} if the following compatibility relations hold: 
\begin{align*}
	\forall x,y,z \in X&\colon \ x \leq y \Longrightarrow x+z\leq y+z,\\
	\forall x,y\in X, \lambda\geq 0&\colon \ x \leq y \Longrightarrow \lambda x\leq \lambda y.
\end{align*}
In this case, we call $X$ a {\it partially ordered vector space}.
A set $K\subseteq X$ is called a {\it cone} if $K$ is closed under addition, $\lambda K\subseteq K$ for every $\lambda \geq 0$, and $K\cap (-K)=\{0\}$. 
If $``\leq"$ is a vector space order on $X$, then  
\begin{equation} 
	\label{eq:cone}
	X_+:=\Set{x \in X\ |\ x\geq 0}
\end{equation} is a cone. On the other hand, if $K\subseteq X$ is a cone, then a vector space order on $X$ is given by \begin{equation}\label{cone_order_relation}
    x\leq y ~:\Longleftrightarrow~  y-x\in K,
\end{equation}
and for $X_+$ as in \eqref{eq:cone},  one obtains $X_+=K$. Hence, there is a one-to-one correspondence between cones and vector space orders, and we denote an ordered vector space by $(X,X_+)$ in case we want to specify the corresponding cone. The articles \cite{cones_and_duality, BARKER1981263, Anke_book} are excellent sources for  information on cones and partially ordered vector spaces.

\begin{definition}
    Let $X$ be a partially ordered vector space with a cone $K$. Let $X'$ denote the (algebraic) dual of $X$. 
    \begin{enumerate}[label=\upshape(\roman*)]
    \item 
        $K$ is said to be {\it generating} if $X=K - K$.
    \item 
        Define the set $K':=\{f\in X'\mid \forall x\in K;~f(x) \geq 0\}$. If $K$ is generating, then $K'$ forms a cone in $X'$, known as the {\it dual cone} of $K$.
    \item   
        An element $x\in K$ is called an {\it extremal} if, for any $0\leq z\leq x$, there exists $\lambda\in [0,1]$ such that $z=\lambda x$. The set of all extremals of $K$ is denoted by $Ext(K)$.
    \item 
        A linear map $T\colon X\to X$ is called {\it positive} if $T[K]\subseteq K.$ The set of all positive operators on $X$, w.r.to $K$, is denoted by $\pi(K)$.
    \item 
        A linear map $T\colon X\to X$ is called a {\it (cone) automorphism} if $T$ is invertible and both $T, T^{-1}\in \pi(K)$. The set of all automorphisms on $K$ is denoted by $Aut(K)$.
    \end{enumerate}
\end{definition}

The following result will be useful in the sequel.
\begin{lemma}\label{lem::inv_image_of_extremal_is_extremal}
    Let $T\colon X\to X$ be an injective, positive linear map. If $T(x)$ is an extremal of $T[K],$ then $x$ is an extremal of $K$. In particular, if $T\in Aut(K)$, then $T[Ext(K)]=Ext(K)$.
\end{lemma}
\begin{proof}
    Since $T$ is positive and injective, it follows that $T[K]$ is also a cone. Let $T(x)$ be an extremal of $T[K]$ and $0\leq z\leq x$. Then $0\leq T(z)\leq T(x)$. Hence, $T(z)=\lambda T(x)$ for some $\lambda\geq 0$, and so $z=\lambda x$. Thus, $x$ is an extremal of $K$, proving the first part. The second part follows by applying the first to both $T$ and $T^{-1}$.
\end{proof}

\subsection{Tensor spaces}

We recall some basic concepts in the theory of symmetric tensors. For a detailed study on symmetric tensors and tensor algebra, refer, for example, to \cite[Chapter 14]{roman2008advanced}. Let $X$ be a vector space and $X\otimes X$ denote the (algebraic) tensor product of $X$ with itself, so that, $X\otimes X = \text{span}\{x\otimes y\ |\ x,y\in X\}$.  If $I$ is a (Hamel) basis of $X$, then $I^2:=\{x\otimes y\ |\ x,y\in I\}$ forms a basis of $X\otimes X$. Thus, for any $A\in X\otimes X$, there exist finitely many $x_i, y_j\in I$ and scalars $a_{ij}\in \R$ such that 
\[
    A=\sum_{i,j} a_{ij}(x_i\otimes y_j).    
\]
Let $[A]$ be the matrix whose $(i,j)^{th}$ entry is $a_{ij}$. Then, the {\it rank} of $A$, denoted by $\rho(A)$, is defined as the rank of the matrix $[A]$. It is well known that the definition of rank is independent of the choice of the basis $I$. Define the map $\sigma\colon X\otimes X\to X\otimes X$ by $\sigma(x\otimes y)=y\otimes x$ and extend to $X\otimes X$ by linearity. An element $A\in X\otimes X$ is called a {\it symmetric tensor} if $\sigma(A)=A$. The space 
\[
    X^{(2)}:=\Set{A\in X\otimes X\ |\ \sigma(A)=A},
\]
consisting of all symmetric tensors, is a subspace of $X\otimes X$, and is called the {\it second symmetric product space} of $X$. Note that $A\in X\otimes X$ is symmetric if and only if $[A]$ is a symmetric matrix. Also, it is well known that $\text{span}\{x\otimes x\colon x\in X\}=X^{(2)}$. Notice that, 
\[
    \Set{A\in X^{(2)} \ |\ \rho(A)\leq 1}=\Set{\lambda (x\otimes x)\ |\ x\in X,~ \lambda\in \R}.
\]
Recall that the rank atmost one elements in $X^{(2)}$ are called the {\it decomposable elements}. 

\begin{proposition}[{\cite[Theorem 2.5]{tensor-anke}}]
    Let $X$ be a partially ordered vector space with a cone $X_+$. Then set $X_+\otimes^\pi X_+$ as in \eqref{eq::projective-cone} is a cone, known as the {\it projective cone}. Also, if $X_+$ is generating, then so is $X_+\otimes^\pi X_+$.
\end{proposition}
Let $X$ be a partially ordered vector space with a cone $X_+$. The restriction $(X_+\otimes^\pi X_+) \cap X^{(2)}$ is a cone in $X^{(2)}$, denoted by $X^{(2)}_+$. Let $A=\sum_{i=1}^n x_i\otimes y_i\in X^{(2)}$. Then 
\[
    A=\frac{1}{2} (A+ \sigma(A))=\frac{1}{2}\left(\sum_i (x_i\otimes y_i) + \sum_i (y_i\otimes x_i)\right)= \frac{1}{2} \sum (x_i\otimes y_i + y_i\otimes x_i).
\]
Thus,
\[
    X^{(2)}_+=\Set{\sum_{i=1}^k (x_i\otimes y_i +y_i\otimes x_i)\ |\ k\in \N, ~x_i, y_i\in X_+}.    
\]
Naturally, the decomposable elements in $X^{(2)}_+$ are called the {\it positive decomposable} elements. Thus, the positive decomposable elements in $X^{(2)}$ are of the form $x\otimes x$ with $x\in X_+$. We refer the reader to \cite{debruyn2022tensorproductsconvexcones} and the references therein for a detailed survey on the theory of cones in tensor product spaces. 
\subsection{Generalized inverses}
Let us recall the notions of the Moore-Penrose inverse and the Drazin inverse 
of a linear operator.
Let $H_1,H_2$ be Hilbert spaces over the same field of scalars and let $\mathcal{B}(H_1,H_2)$ denote the space of all bounded linear operators from $H_1$ to $H_2$.  Let $\Range(T),~ \ker(T)$ denote the range and kernel of $T$, respectively. If $T\in \mathcal{B}(H_1,H_2)$ has closed range, then there exists a unique (\cite[Theorem 2.2.1]{Groetsch}) linear operator $U\in \mathcal{B}(H_2,H_1)$ such that 
    \begin{equation}\label{MP_equations}
        TUT=T,~ UTU=U, ~ (TU)^*=TU, ~ (UT)^*=UT.
    \end{equation}
This operator $U$ is called the {\it Moore-Penrose inverse} of $T$, denoted by $T^\dagger$. On the other hand, if, for $T\in \mathcal{B}(H_1,H_2)$, there exists $U\in \mathcal{B}(H_2,H_1)$ satisfying \eqref{MP_equations}, then $\Range(T)$ is closed. To see this, Let $\{y_n=T(x_n)\}_{n\in \N}\subseteq \Range(T)$ be such that $y_n\to y.$ Then $UT(x_n)\to U(y)$, and so $y_n=TUT(x_n)\to TU(y)$. Thus $y=TU(y)\in \Range(T)$. For more details on the Moore-Penrose inverse, see \cite{Ben_israel_book,Groetsch}.

For the Drazin inverse, we consider a vector space $X$ and a linear operator $T\colon X\rightarrow X$. We say that $T$ is Drazin invertible if there exists $S\colon X\rightarrow X$ and $k\in \Z_+:=\N\cup\{0\}$ such that \begin{equation}\label{eq::drazin_prelim}
        T^kST=T^k,\quad STS=S,\quad TS=ST,
    \end{equation} 
    If such an $S$ exists, then it is unique, called the Drazin inverse of $T$, and is denoted by $T^D$. Also, the least such $k\in \Z_+$ is called the {\it index} of $T$, denoted by $\text{Ind}(T)$. If there does not exist any such $k$, then we set $\text{Ind }T=\infty$. It is well known (cf. \cite[Theorem 4]{Note-Drazin}, or \cite[Theorem 2.25]{Kalauch2026-pg} for a short proof) that $\text{Ind }T\leq k$ if and only if $\Range(T^{k+1})=\Range(T^k)$ and $\ker(T^{k+1})=\ker(T^k)$. The Drazin inverse of a linear operator on a finite dimensional vector space always exists. For more details, we refer the reader to \cite{Ben_israel_book, gen_inv_line_transformation}.

\section{Linear preservers of positive decomposable elements}\label{sec::positive-decompo-preservers}

Let $(X,X_+)$ be a partially ordered vector space. Let $\mathcal{D}:=\{\lambda(x\otimes x)\ |\ x\in X, \;\lambda\in \R\}$ denote the set of decomposable elements of $X^{(2)}$. Then, a linear map $T\colon X^{(2)}\to X^{(2)}$ is rank one non-increasing if and only if $T[\calD]\subseteq \calD$. Recall that a representation theorem for such maps is provided in Theorem~\ref{thm::Lim}. Suppose $X_+$ is a generating cone. Let $\calD_+:=\calD\cap X^{(2)}_+$ denote the set of all positive decomposable elements of $X^{(2)}$. Notice that $\calD_+=\{x\otimes x\ |\ x\in X_+\}$.  In this section, our aim is to characterize linear maps $T\colon X^{(2)}\to X^{(2)}$ that preserve the set $ \calD_+$, by employing Theorem \ref{thm::Lim}. For any matrix $X$ and indices $i,j,k,l$, we let $X\{i,j|k,l\}$ denote the submatrix of $X$ with the rows $i, j$ and columns $k, l$.

\begin{theorem}\label{thm::Lim_2.0}
    Let $X$ be a partially ordered vector space with a generating cone $X_+$. Let $T\colon X^{(2)}\to X^{(2)}$ be a linear map. Then, $T[\calD_+]\subseteq \calD_+$ if and only if either of the following holds.
    \begin{enumerate}[label=\upshape(\roman*)]
        \item \label{(i)}
            $T= P_2(f)$ for some linear map $f\colon X\to X$ with the property that $f[X_+]\subseteq X_+\cup (-X_+)$.
        \item \label{(ii)}
            $T(\cdot)=\phi(\cdot) B$ for some $B\in \calD_+$ and $\phi\in (X^{(2)})'$ such that $\phi[\calD_+]\subseteq [0,\infty)$.
    \end{enumerate}
\end{theorem}
\begin{proof}
    Sufficiency is easily verified. Let $T[\calD_+]\subseteq \calD_+$. First, we show that $T$ is rank one non-increasing, i.e., $T[\calD]\subseteq \calD$.  Let $A\in X^{(2)}$ be such that $\rho(A)=1$. Then $A=\lambda (u\otimes u)$ for some $u\in X$ and $\lambda\in \R\setminus\{0\}$. Since $\rho(\lambda (u\otimes u))=\rho(u\otimes u)$, without loss of generality, assume $\lambda=1$. Since $X_+$ is generating, there exists $u_1, u_2\in X_+$ such that $u=u_1-u_2$. For $\alpha\in \R$, define $u(\alpha):=u_1+\alpha u_2$. Then, for each $\alpha\geq 0$, we have $u(\alpha)\in X_+$. Define $A_\alpha:=u(\alpha)\otimes u(\alpha)\in X^{(2)}$. Notice that $A_{-1}=A$. Observe that 
    \[
        A_\alpha=u_1^2+\alpha^2 u_2^2+\alpha(u_1\otimes u_2+ u_2\otimes u_1).    
    \]
    By hypothesis, for each $\alpha\geq 0$, $\rho(T(A_\alpha))\leq 1$. Hence the determinant of any submatrix of $T(A_\alpha)$ of order $2$ is $0.$ Fix $1\leq i<j$ and $1\leq k<l$. Then, one has $\det ([T(A_\alpha)]\{i,j|k,l\})=0$ for every $\alpha\geq 0$. Since 
    \[
        [T(A_\alpha)]=[T(u_1^2)]+\alpha^2 [T(u_2^2)] + \alpha [T(u_1\otimes u_2+u_2\otimes u_1)],    
    \]
    we have $\det ([T(A_\alpha)]\{i,j|k,l\})$ is a polynomial in $\alpha$ with real coefficients, that is zero on $[0, \infty)$. Hence 
    \[
        \det ([T(A_\alpha)]\{i,j|k,l\})=0 \quad \forall \alpha\in \R.
    \]
    In particular, for $\alpha=-1$. Since $i,j,k,l$ are arbitrary, it follows that $\rho(T(A))\leq 1$. Hence, Theorem~\ref{thm::Lim} applies, and so we only have to check for the additional conditions on $f,$ $\phi$ and $y$. 
    
    Let $T$ be of type \ref{(i)} in Theorem~\ref{thm::Lim} so that $T=cP_2(f)$ for some $c\in \R\setminus\{0\}$ and a linear map $f\colon X\to X$. First observe that, for $\alpha\in \R$ and $0\neq y\in X$, one has $\alpha(y\otimes y)\in \calD_+$ if and only if $\alpha\geq 0$ and $y\in X_+\cup (-X_+)$. For, let $\alpha(y\otimes y) \in \calD_+$, so that $\alpha(y\otimes y)=u\otimes u$ for some $u\in X_+$. Then, by the equality of tensor products, we get $y=\lambda u$ for some $\lambda\in \R$. Thus, $y\in X_+\cup (-X_+)$. Now, if $\alpha<0$, then 
    \[
        \alpha(y\otimes y)\in \calD_+\cap (-\calD_+)\subseteq X^{(2)}_+\cap (-X^{(2)}_+)=\{0\},
    \]
    a contradiction to $y\neq 0$ and so $\alpha \geq 0$. Since $T[\calD_+]\subseteq \calD_+$, for $x\in X_+$, we have 
    $$c (f(x)\otimes f(x))=T(x\otimes x) \in \calD_+.$$ Hence, by the above argument, we get $c> 0$ and $f(x)\in X_+\cup (-X_+)$.  Now, since $c(f(x)\otimes f(x))=(\sqrt{c}f)(x)\otimes (\sqrt{c}f)(x)$ and $\sqrt{c}f\colon X\to X$ is again linear, without loss of generality, we assume that $c=1$.
    
    Let $T$ be of type \ref{(ii)} in Theorem \ref{thm::Lim}, so that, for $A\in X^{(2)}$, we have $T(A)=\phi(A)y\otimes y$ for some linear functional $\phi\colon X^{(2)}\to \R$ and $y\in X$. Let $x\in X_+$. Since $T[\calD_+]\subseteq \calD_+$, it follows that $T(x\otimes x)=\phi(x\otimes x) y\otimes y\in \calD_+$, and so $y\in X_+\cup(-X_+)$ and $\phi(x\otimes x)\geq 0$. Thus, we get $B:=y\otimes y\in \calD_+$ and, since $x$ is arbitrary, we get that $\phi$ is positive on $\calD_+$.
\end{proof}

\begin{note}
    In the proof above, we first showed that if $T$ preserves $\calD_+$, then it also preserves $\calD$. This idea is motivated by the proof of \cite[Lemma 3.1]{Beasley2019-sp}, where the particular situation of $X:=\R^n$ and $X_+:=\R^n_+$ is considered.
\end{note}

Next, we give a characterization of linear maps $f\colon X\to X$ satisfying the hypothesis (i) of Theorem \ref{thm::Lim_2.0}, viz., $f(x)\in X_+\cup (-X_+)$, for every $x\in X_+$. We state and prove an intermediate lemma.

\begin{lemma}\label{lemma::boundary_lemma}
    Let $A\subseteq X$ be a closed and convex subset in a normed vector space $X$. Let $u\in A$ and $v\notin A$. Consider the map $g\colon [0,1]\to X$, defined as $g(t)=(1-t)u+tv$ for $t\in [0,1]$. Then there exists $c\in [0,1)$ such that $g(t)\in A$ if and only if $t\in [0,c].$
\end{lemma}
\begin{proof}
    First, observe that $g$ is continuous. Set $c:=\inf g^{-1}(A^c)$. Clearly, $g([0,c])\subseteq A$. Since $A^c$ is an open set, $g^{-1}(A^c)$ is open. Note that $1 \in g^{-1}(A^c).$ Thus $(\delta,1] \subseteq g^{-1}(A^c)$ for some $\delta>0$. Hence, $c<1$. Notice that, for any $t\in (0,1)$ and $\lambda\in (0, t)$, $g(\lambda)=(1-\frac{\lambda}{t})u+\frac{\lambda}{t}g(t)$. Since $A$ is convex and $u\in A,$ if $g(t)\in A$, then $g((0,t))\subseteq A$. Thus, we have shown that $g((c,1])\subseteq A^c$, completing the proof.
\end{proof}
\begin{remark}\label{rem::boundary-lemma}
    Observe that, in Lemma~\ref{lemma::boundary_lemma}, $g$ is a directed line segment from $u$ to $v$. If we change the orientation of $g$, i.e., define $g(t):=tu + (1-t)v$ for $t\in [0,1]$, then there exists $c\in (0,1]$ such that $g(t)\in A$ if and only $t\in [c,1]$.
\end{remark}

\begin{theorem}\label{thm::prop(U)_char}
    Let $X$ be a normed vector space and $X_+\subseteq X$ be a closed and generating cone. A linear map $f\colon X\to X$ satisfies $f[X_+]\subseteq X_+\cup(-X_+)$ if and only if one of the following holds:
    \begin{enumerate}[label=\upshape(\roman*)]
        \item\label{unisignes:i}
            $f\in \pi(X_+)\cup (-\pi(X_+))$ 
        \item 
            there exists $u\in X_+\cup(-X_+)$ and $\psi\in X'$ such that $f(x)=\psi(x) u$, for every $x\in X$.
    \end{enumerate}
\end{theorem}
\begin{proof}
    Sufficiency is easily verified. Suppose that $f$ satisfies the assumption and $f\notin \pi(X_+)\cup (-\pi(X_+))$. Hence there exist $x,y\in X_+$ such that $0\neq f(x)\in X_+$ and $0\neq f(y)\in -X_+$. Then, $x$ and $y$ are linearly independent. For, let $x=\lambda y$ for some $\lambda\in \R$. Since $x,y\in X_+$, we get $\lambda\geq 0$, and thus $f(x)=\lambda f(y)\in X_+\cap (-X_+)=\{0\}$, a contradiction. Since $X_+$ is generating, we have $X=X_+-X_+=\rm span(X_+).$ Let $\calB\subseteq X_+$ be a Hamel basis of~$X, ~\calB_1:=\{z\in \calB\colon f(z)\in X_+\}$ and $\calB_2:=\{z\in \calB\colon f(z)\in -X_+\}$. Then $x\in \calB_1$ and $y\in \calB_2$. Our claim is that $f$ is a rank one operator. Suppose to the contrary that there exist at least two linearly independent vectors in $P[\calB]$. We claim that there exist $z_1\in \calB_1$ and $z_2\in \calB_2$ such that $\{f(z_1), f(z_2)\}$ is linearly independent. For, if $f(x)$ and $f(y)$ are independent, then this holds. Otherwise, let $u\in \calB$ be such that $Pu$ and $Px$ are independent. If $u\in \calB_1$, then choose $z_1:=u$ and $z_2:=y$. If $u\in \calB_2$, then choose $z_1:=x$ and $z_2:=u$.
    
    Consider the map $g\colon [0,1]\to X$, given by $g(t):=(1-t)z_1+tz_2$ for $t\in [0,1]$. Since $z_1, z_2\in X_+$, $g([0,1])\subseteq X_+$, and so $(f\circ g)([0,1])\subseteq X_+\cup (-X_+)$. Thus, for any $t\in [0,1],$
    \begin{equation}\label{eq::3.1}
        (f\circ g)(t) = (1-t)f(z_1) + t f(z_2)\in X_+\cup(-X_+).
    \end{equation}
     Also, $(f\circ g)(0)=f(z_1)\in X_+$ and $(f\circ g)(1)=f(z_2)\in -X_+$. So, by  Lemma~\ref{lemma::boundary_lemma} and Remark~\ref{rem::boundary-lemma}, there exist $c_1\in [0,1)$ and $c_2\in (0,1]$ such that $(f\circ g)(t)\in X_+$ if and only if $t\in [0,c_1]$ and $(f\circ g)(t)\in -X_+$ if and only if $t\in [c_2,1]$. Suppose that $[0,c_1]\cap [c_2, 1]\neq \emptyset$ and let $c\in [0,c_1]\cap [c_2, 1]$. Then $f\circ g(c)\in X_+\cap (-X_+)=\{0\}$, implying that $\{f(z_1), f(z_2)\}$ is linearly dependent, a contradiction. Thus, $[0,c_1]\cap [c_2, 1]=\emptyset$ and so $c_1<c_2$. This, in turn, implies that $f\circ g(t)\notin X_+\cup (-X_+)$ for all $t\in (c_1,c_2)$, a  contradiction to \eqref{eq::3.1}. Therefore, $f$ is a rank one operator, and hence there exists $\psi\in X'$ and $u\in X$ such that $f(x)=\psi(x) u$ for every $x\in X.$ It is clear that $u\in X_+\cup(-X_+)$. This completes the proof.
\end{proof}

\section{The group of automorphisms of the completely positive cone}\label{Sec::Aut(CP)}

Let $X$ be a partially ordered normed vector space with a closed and generating cone $X_+.$ Consider the cone 
\[
    CP(X_+):=\Set{\sum_{i=1}^n u_i\otimes u_i\ |\ u_i\in X_+,~ n\in \N}\subseteq X^{(2)}_+,
\]
known as the {\it completely positive cone}. In this section, applying Theorem \ref{thm::Lim_2.0}, we provide a characterization of $Aut(CP(X_+))$.


\begin{proposition}\label{ext-CP-K}
    $Ext(CP(X_+))=\{u\otimes u\ |\ u\in X_+\}$.
\end{proposition}
\begin{proof}
    The forward inclusion is straightforward. Let $u\in X_+$ and $A\in CP(X_+)$ such that  $0\leq A\leq u\otimes u$. Then $u\otimes u -A\in CP(X_+)$, so that there exist $u_1,\ldots,u_n\in X_+$ such that $u\otimes u = A + \sum_i u_i\otimes u_i$. Since $A\in CP(X_+)$, without loss of generality, assume that there exists $k\in \N$ and $u_i\in X_+$ such that
    \[
        u\otimes u = \sum_{i=1}^k u_i\otimes u_i.
    \]
    For each $f\in X'$, define $\psi_f\colon X\otimes X\to \R$ as $\psi_f(x\otimes y)=f(x)\otimes f(y)$. Then
    \[
        f(u)^2=\psi_f(u\otimes u)=\psi_f\left(\sum_{i=1}^k u_i\otimes u_i\right)=\sum_{i=1}^k f(u_i)^2.
    \]
     Thus, $f(u)=0$ implies $f(u_i)=0$ for all $i$. Hence, for each $i$, we get $\{u,u_i\}$ is linearly dependent. This proves that $u\otimes u\in Ext(CP(X_+))$.  
\end{proof}

\begin{theorem}\label{thm::Aut-CP}
Let $X$ be a partially ordered normed vector space with a closed and generating cone $X_+$. Consider the space $X^{(2)}$ and the cone $CP(X_+)$. Then the following are equivalent:
    \begin{enumerate}[label=\upshape(\roman*)]
        \item \label{Aut-CP-(i)}
            $T\in Aut(CP(X_+))$,
        \item \label{Aut-Cp-(ii)}
            $T=P_2(f)$ for some $f\in Aut(X_+)$.
    \end{enumerate}
\end{theorem}
\begin{proof}
    The implication \ref{Aut-Cp-(ii)} $\implies$ \ref{Aut-CP-(i)} is easily verified. Let $T\in Aut(CP(X_+)).$ Notice that, by Proposition~\ref{ext-CP-K}, $Ext(CP(X_+))=\calD_+$. Thus, by Lemma~\ref{lem::inv_image_of_extremal_is_extremal}, both $T$ and $T^{-1}$ preserve the set $\calD_+$. Hence, by Theorem \ref{thm::Lim_2.0}, there exists $f\colon X\to X$ such that $f[X_+]\subseteq X_+\cup (-X_+)$ and $T=P_2(f).$ Bijectivity of $T$ implies bijectivity of $f$. By Theorem \ref{thm::prop(U)_char}, it follows that $f, f^{-1}\in \pi(X_+)\cup (-\pi(X_+))$. Since $ff^{-1}$ is identity, without loss of generality, assume that both $f, f^{-1}\in \pi(X_+)$, so that $f\in Aut(X_+).$ This proves \ref{Aut-CP-(i)} $\implies$ \ref{Aut-Cp-(ii)}.
\end{proof}

Recall that $M_n$ denotes the set of all $n\times n$ real matrices and $S^n$ denotes the set of all $n\times n$ real symmetric matrices. Also, we say that $A\in S^n$ is {\it copositive} if $\langle Ax,x\rangle\geq 0$ for all $x\in \R^n_+$, and {\it strictly copositive} if $\langle Ax,x\rangle > 0$ for all $x\in \R^n_+\setminus \{0\}$.

\begin{remark}\label{rem::S^n}
    Let $X:=\R^n$. Then $X\otimes X=M_n$ and $X^{(2)}=S^n$. Let $P\colon \R^n\to \R^n$ be a linear map. For $A\in S^n$, let $A=\sum \lambda_i u_iu_i^\top$ for some finitely many $\lambda_i\in \R$ and $u_i\in \R^n$. Hence, here $P_2(P)\colon S^n\to S^n$ is given by 
    \[
        P_2(P)(A)=\sum \lambda_i Pu_i (Pu_i)^\top= \sum  P(\lambda_i u_iu_i^\top) P^\top= P A P^\top.
    \]
    Also, using the Riesz representation theorem, we identify $(S^n)'=S^n$.
\end{remark}
Next, by appealing to Theorem~\ref{thm::Aut-CP}, we obtain the main result of \cite{GOWDA20133862}.

\begin{corollary}[{\cite[Theorem 1]{GOWDA20133862}}]\label{Ex::COP}
    Let $K\subseteq \R^n$ be a closed and generating cone and 
    \[
        CP(K):=\Set{\sum_{i=1}^k u_iu_i^\top\ |\ k\in \N,~u_i\in K}    
    \]
    be the cone of {$K$-completely positive matrices}. Let $T\colon S^n\to S^n$ be a linear map. Then the following are equivalent.
    \begin{enumerate}[label=\upshape(\roman*)]
        \item
            $T\in Aut(CP(K))$,
        \item
            there exists $P\in Aut(K)$ such that $T(A)=PAP^\top$ for every $A\in S^n$.
    \end{enumerate}
\end{corollary}
\begin{proof}
    Using Remark~\ref{rem::S^n}, the result is an immediate consequence of substituting $X:=\R^n$ and $X_+:=K$ in Theorem~\ref{thm::Aut-CP} .
\end{proof}

\section{CP-rank preservers}\label{sec::CP-rank}

Let $X$ be a partially ordered vector space with a cone $X_+$ and $A\in CP(X_+)\setminus \{0\}$. Then the {\it CP-rank} of $A$, denoted by $cp(A)$, is defined as the least integer $k$ such that 
\[
    A=\sum_{i=1}^k u_i\otimes u_i, \quad  u_i\in X_+.
\]
For $k\in \N$, the collection of all matrices of CP-rank $k$ is denoted by $CP_k$. Notice that 
\[
    CP_1=\{x\otimes x\mid x\in X_+\setminus\{0\}\}=\calD_+\setminus \{0\}.
\]
For example, in the particular case of $X:=\R^n$ and $X_+:=\R^n_+$ we know $X^{(2)}=S^n$. Here, a non-zero matrix $A\in S^n$ has CP-rank-$k$ if $k$ is the least integer such that $A=BB^\top$ with $B\in \R^{n\times k}_+$ and $k\in \N$. 
In \cite{Beasley2019-sp}, the authors study linear maps $T\colon S^n\to S^n$ that preserve the set $CP_1$ for the standard cone $\R^n_+$. In the next result, we prove a generalization to the second symmetric space $X^{(2)}$.

\begin{theorem}\label{Thm::CP-rank}
    Let $X$ be a ordered normed vector space with a closed and generating cone $X_+$. Let $T\colon X^{(2)}\to X^{(2)}$ be a linear map. Then $T$ preserves $CP_1$ if and only if either of the following holds:
    \begin{enumerate}[label=\upshape(\roman*)]
        \item 
            $T=P_2(f)$ for some linear map $f\in \pi(X_+)$ such that $f(x)\neq 0$ for every $x\in X_+\setminus\{0\}$.
        \item 
            $T(\cdot)=\phi(\cdot) B$ for some $B\in CP_1$ and $\phi\in (X^{(2)})'$ such that $\phi[\calD_+\setminus\{0\}]\subseteq (0,\infty)$.
    \end{enumerate}
\end{theorem}
\begin{proof}
    Sufficiency is easily verified. Let $T$ preserves $CP_1$. Since $CP_1=\calD_+\setminus\{0\}$, it follows that $T[\calD_+]\subseteq \calD_+$. Hence Theorem~\ref{thm::Lim_2.0} applies. 
    
    Let $T$ be of type (i) in Theorem~\ref{thm::Lim_2.0}, so that $T=P_2(f)$ with $f\colon X\to X$ satisfying the property $f[X_+]\subseteq X_+\cup(-X_+)$. By Theorem~\ref{thm::prop(U)_char}, there are two possibilities for $f.$ Let $f$ be of form (i) in Theorem~\ref{thm::prop(U)_char}, i.e. $f\in \pi(X_+)\cup (-\pi(X_+))$. Since $P_2(f)=P_2(-f)$, without loss of generality, we assume that $f\in \pi(X_+).$ If $f(x)=0$ for some $x\in X_+\setminus\{0\}$, then $T(x\otimes x)=f(x)\otimes f(x)=0$, a contradiction to the fact that $T[CP_1]\subseteq CP_1.$ Hence $f(x)\neq 0$ for every $x\in X_+\setminus \{0\}$. Suppose $f$ is of type (ii) in  Theorem~\ref{thm::prop(U)_char}, i.e., there exists $u\in X_+\cup (-X_+)$ and $\psi\in X'$ such that $f(x)=\psi(x) u$ for $x\in X$. If $\psi\in X_+'\cup(-X_+')$, then $f\in \pi(X_+)\cup (-\pi(X_+)$. This is already considered. Hence, suppose $\psi\notin X'_+\cup (-X'_+)$. Then there exist $x_1,x_2\in X_+$ such that $\psi(x_1)>0$ and $\psi(x_2)<0$. Set 
    \[
        x:=\frac{x_1}{\psi(x_1)}- \frac{x_2}{\psi(x_2)}\in X_+\setminus\{0\},    
    \]
    then $\psi(x)=0$. Hence $T(x\otimes x)=f(x)\otimes f(x)=0$, a contradiction to $T[CP_1]\subseteq CP_1$.

    Let $T$ be of type (ii) in Theorem~\ref{thm::Lim_2.0}, so that, for $A\in X^{(2)}$, we have $T(A)=\phi(A) B$ for some $B\in \calD_+$ and $\phi\in (X^{(2)})'$ satisfying $\phi[\calD_+]\subseteq [0,\infty).$  Clearly, $B\neq 0$, and hence $B\in CP_1.$  If $\phi(A)=0$ for some $A\in \calD_+\setminus\{0\}$, then $T(A)=0$. A contradiction to $T[CP_1]\subseteq CP_1$. This proves the result.
\end{proof}

Next, by appealing to Theorem~\ref{Thm::CP-rank}, we obtain \cite[Theorem 3.1]{Beasley2019-sp} as a corollary.

\begin{corollary}\label{cor::Cp-rank-1}
    Let $T\colon S^n\to S^n$ be a linear map. Then $T$ preserves $CP_1$ (w.r.to $\R^n_+$) if and only if either: $T(A)=\tr(AU)B$ for some strictly copositive matrix $U$ and some completely positive matrix $B$ of CP-rank $1$; or there exists an entry-wise nonnegative matrix $P\in M_n(\R)$ with no zero column such that $T(A)=PAP^\top$ for all $A\in S^n$.
\end{corollary}
\begin{proof}
      Notice that, $\pi(\R^n_+)$ is the set of all entry-wise nonnegative matrices and, for $P\in \pi(\R^n_+)$, we have $P(x)\neq 0$ for all $x\in \R^n_+\setminus\{0\}$ if and only if $P$ has no zero column. Thus, by substituting $X:=\R^n$ and $X_+:=\R^n_+$ in Theorem~\ref{Thm::CP-rank} and using Remark~\ref{rem::S^n}, it follows that $T$ preserves $CP_1$ if and only if, for every $A\in S^n,$ either of the following holds:
    \begin{enumerate}[label=\upshape(\roman*)]
        \item 
            $T(A)=PAP^\top$, where $P$ is an entry-wise nonnegative matrix with no zero column.
        \item   
            $T(A)=\tr(AU) B$, where $B\in CP_1$ and $U\in S^n$ such that 
            \[
                \tr(xx^\top U)=x^\top Ux>0, \quad \forall x\in \R^n_+\setminus\{0\}.
            \]
    \end{enumerate}  
    This proves the result.
\end{proof}

\section{Generalized inverses of rank one non-increasing maps}\label{sec::gen-inve}

\subsection{Moore-Penrose inverse}

Let $(X, \langle \cdot,\cdot\rangle)$ be a Hilbert space. Then, the completion of $X\otimes X$, with the inner product given by $\langle x\otimes y,u\otimes v\rangle=\langle x,u\rangle \langle y,v\rangle$, is a Hilbert space. Without loss of generality, we consider $X\otimes X$ as a Hilbert space. For $S\subseteq X\otimes X$, $\overline{S}$ denotes the topological closure of $S.$ Note that, here 
\[
    X\otimes X=\overline{\text{span}\{x\otimes y\colon x,y\in X\}}.    
\]
Same as earlier, here $X^{(2)}:=\{A\in X\otimes X\colon \sigma(A)=A\}$, where $\sigma\colon X\otimes X\to X\otimes X$ is defined as $\sigma(x\otimes y)=y\otimes x$, for every $x,y\in X$, and extended linearly. Since $\sigma$ is continuous, $X^{(2)}$ is closed and, so, a Hilbert space. We start with an example showing that the Moore-Penrose inverse of a rank one non-increasing map does not share that property, in general.

\begin{example}\label{ex::counter-ex-MP-inverses}
    Let $X:=\R^n$, and so $X^{(2)}=S^n$. Consider the map $T\colon S^n\to S^n$ given as
    \[
        T(A)=\langle A, I\rangle ee^\top = \tr(A) ee^\top,\quad \forall A\in S^n,
    \]
    where $I\in S^n$ denotes the identity matrix. It is clear that $T$ is rank one non-increasing. It can be easily verified that $T^\dagger$ is given by
    \[
        T^\dagger(A)=\frac{1}{n^3}(e^\top A e) I, \quad \forall A\in S^n.
    \]
     Since $\Range(T^\dagger)=\text{span}(I)$, it follows that $T^\dagger$ is not rank one non-increasing.
\end{example}

In the next two results, we provide a characterization for linear maps $T$ such that both $T$ and $T^\dagger$ are rank one non-increasing.

\begin{proposition}\label{P-2(f)-MP}
    Let $T\colon X^{(2)}\to X^{(2)}$ be a linear map given by $T=cP_2(f)$. Then $T^\dagger$ exists if and only if $f^\dagger$ exists, and in that case $T^\dagger=\frac{1}{c}P_2(f^\dagger)$. 
\end{proposition}
\begin{proof}
    Let $T^\dagger$ exist\sout{s}. It is enough to show that $\Range(f)$ is closed. Let $y_n=f(x_n)\in \Range(f)$ be such that $y_n\to y\in X$. First, we show that $y_n\otimes y_n\to y\otimes y$. Consider
    \begin{align*}
        \|y_n\otimes y_n - y\otimes y\|^2 &= \langle y_n\otimes y_n - y\otimes y, y_n\otimes y_n - y\otimes y\rangle\\
        &=\langle y_n\otimes y_n, y_n\otimes y_n\rangle + \langle y\otimes y, y\otimes y\rangle - 2 \langle y_n\otimes y_n, y\otimes y\rangle\\
        &=\|y_n\|^4 + \|y\|^4 - 2 | \langle y_n, y\rangle|^2\xrightarrow[n\to \infty]{} \|y\|^4 + \|y\|^4 - 2 \|y\|^4 =0.
    \end{align*}
    Since $y_n\otimes y_n = \frac{1}{c} T(x_n\otimes x_n)\in \Range(T)$, it follows that $y\otimes y \in \Range(T)$. Thus $y\otimes y= T(A)$ for some $A\in X^{(2)}$. Let $A=\sum_{i=1}^m \lambda_i (x_i\otimes x_i)$. Hence 
    \[
        y\otimes y= c \left(\sum_{i} \lambda_i (f(x_i)\otimes f(x_i))\right)= \sum (\lambda_i f(x_i))\otimes (c f(x_i))
    \]
     and so, by the equality of tensor products, we get $y\in \text{span}\{f(x_i)\mid 1\leq i\leq m\}\subseteq \Range(f)$.
   Next, we claim that $T^\dagger = \frac{1}{c} P_2(f^\dagger).$ The proof is by verification.  Let $X:=\frac{1}{c}P_2(f^\dagger)$. Then
    \[
        TX(x\otimes x)= ff^\dagger (x)\otimes ff^\dagger (x)
    \]
   and so
   \[
        TXT(x\otimes x)= TX(c f(x)\otimes f(x))=c (ff^\dagger f(x)\otimes ff^\dagger f(x))=T(x\otimes x)
   \]
    and 
    \[
        XTX(x\otimes x)=X(ff^\dagger (x)\otimes ff^\dagger (x))=\frac{1}{c}(f^\dagger ff^\dagger(x)\otimes f^\dagger f f^\dagger(x))=X(x\otimes x).
    \]
    Notice that
    \begin{align*}
        \langle TX(x\otimes x), y\otimes y\rangle &= \langle ff^\dagger(x)\otimes ff^\dagger(x), y\otimes y\rangle\\
        &= |\langle ff^\dagger(x), y\rangle|^2=|\langle x, ff^\dagger(y)\rangle|^2\\
        &=\langle x\otimes x, ff^\dagger(y)\otimes ff^\dagger(y)\rangle\\
        &=\langle x\otimes x,TX(y\otimes y)\rangle.
    \end{align*}
    Since $\overline{\text{span}\{x\otimes x\colon x\in X\}}=X^{(2)}$, it follows that $(TX)^*=TX$. A similar proof applies for $(XT)^*=XT$ and is skiped. Thus $T^\dagger=\frac{1}{c}P_2(f^\dagger)$. 
\end{proof}

\begin{theorem}\label{thm::rank-non-increasing-Moore-Penrose-inverse}
    Let $T\colon X^{(2)}\to X^{(2)}$ be a linear rank one non-increasing map. If $T^\dagger$ exists, then $T^\dagger$ is rank one non-increasing if and only if either of the following holds.
    \begin{enumerate}[label=\upshape(\roman*)]
        \item \label{form(i)} 
            $T=cP_2(f)$ for some linear map $f\colon X\to X$ with closed range and $c\in \R\setminus \{0\}$ .
        \item \label{form(ii)} 
            $T(\cdot)=\langle \cdot, U\rangle y\otimes y$ for some $y\in X$ and $U\in \calD$.
    \end{enumerate}
\end{theorem}
\begin{proof}
    Let $T$ be of form \ref{form(i)}. Then by Proposition~\ref{P-2(f)-MP}, $$T^\dagger=\frac{1}{c}P_2(f^\dagger).$$ If $T$ is of form \ref{form(ii)}, then, it can be verified that, $T^\dagger$ is given by 
    \[
        T^\dagger(A)=\frac{1}{\|U\|^2 \|y\otimes y\|^2}\langle A, y\otimes y\rangle U,\quad \forall A\in X^{(2)}.
    \]
    In either case, by Theorem~\ref{thm::Lim}, $T^\dagger$ is rank one non-increasing. 
    
    Conversely, suppose $T^\dagger$ is rank one non-increasing. Recall that
    \[
        \{A\in X^{(2)}\mid \rho(A)\leq 1\}=\Set{\lambda(x\otimes x)\ |\ x\in X,\, \lambda\in \R}=\calD.
    \]
    Then, by Theorem~\ref{thm::Lim}, there are only two possibilities for $T$: either $T=cP_2(f)$ for some linear map $f\colon X\to X$ and $c\in \R\setminus \{0\}$ ; or $T(A)=\phi(A) y\otimes y$ for some linear functional $\phi\colon X^{(2)}\to \R$ and $y\in X$. In the former case, by Theorem~\ref{P-2(f)-MP}, $\Range(f)$ is closed. In the latter case, since $X^{(2)}$ is a Hilbert space, by the Riesz representation theorem, there exists $U\in X^{(2)}$ such that $\phi(A)=\langle A, U\rangle$ for all $A\in X^{(2)}$. Thus, $T^\dagger$ is given by $T^\dagger(A)=\frac{1}{\|U\|^2 \|y\otimes y\|^2}\langle A, y\otimes y\rangle U$, for $A\in X^{(2)}$. Hence $\Range(T^\dagger)=\text{span}\{U\}$ and so it follows that $U\in \calD$. This completes the proof.
\end{proof}
Notice that, the operator $T$ in Example~\ref{ex::counter-ex-MP-inverses} is neither of form~\ref{form(i)} nor of form~\ref{form(ii)} of Theorem~\ref{thm::rank-non-increasing-Moore-Penrose-inverse}.

\begin{example}
    Consider the space $\ell^2(\N):=\{x\colon \N\to \R\ |\ \sum x_i^2<\infty\}$ of all square summable real sequences. It is well known (cf. \cite[Example 2.6.11]{kadison1997fundamentals}) that 
    \[
        \ell^2\otimes \ell^2\cong \ell^2(\N\times \N)=\Set{ a\colon \N\times \N \ |\ \sum_{i,j} a(i,j)^2<\infty}
    \]
    with the following identification. For $x,y\in \ell^2$, $x\otimes y\colon \N\times \N\to \R$ is given by $(x\otimes y)(i,j)=x_iy_j$ for all $i,j\in \N$. Observe that each element of $\ell^2(\N\times \N)$ can be thought of as an infinite matrix. Notice that, here 
    \[
        (\ell^2)^{(2)}=\Set{a\in \ell^2(\N\times \N)\ |\ \forall (i,j)\in \N\times \N\colon a(i,j)=a(j,i)}.
    \] In that case, the decomposable elements has the following representation. For $x\in \ell^2,$ 
    \[
        x\otimes x =\begin{bmatrix}
            x_1&x_1&x_1&\ldots\\
            x_2&x_2&x_2&\ldots\\
             \vdots & \vdots& \vdots &\\
             x_n& x_n &x_n&\ldots\\
             \vdots & \vdots& \vdots & \ddots
        \end{bmatrix} \odot_H  \begin{bmatrix}
            x_1&x_2&\ldots & x_n &\ldots\\
            x_1&x_2&\ldots & x_n &\ldots\\
             \vdots & \vdots &\\
             x_1& x_2 &\ldots & x_n &\ldots\\
             \vdots & \vdots & & \vdots& \ddots
        \end{bmatrix},
    \]
    where $\odot_H$ is the Hadamard entry-wise product. 
    Now, consider the right shift operator $f\colon \ell^2\to \ell^2$ given by 
    \[
        f(x_1,x_2,\ldots, x_n, \ldots)=(0, x_1, x_2,\ldots, x_n, \ldots), \quad \forall (x_1,x_2,\ldots, x_n, \ldots)\in \ell^2.
    \]
    Then, $\Range(f)$ is closed and $f^\dagger\colon \ell^2\to \ell^2$ is the left shift operator, given by $f^\dagger(x)=(x_{i+1})_{i\in \N}$. Then $P_2(f)\colon (\ell^2)^{(2)}\to (\ell^2)^{(2)}$ is given by  
    \[
        P_2(f)(x\otimes y)(i,j)=(f(x)\otimes f(y))(i,j)=\begin{cases}
            0,& i=1 ~\text{or}~  j=1,\\
            x_{i-1}y_{j-1}, &\text{otherwise}.
        \end{cases}
    \]
    In the matrix form, the action of $P_2(f)$ on rank one elements is as follows.
    \[
        x\otimes x\longmapsto \begin{bmatrix}
            0&0&0&\ldots\\
            x_1&x_1&x_1&\ldots\\
             \vdots & \vdots& \vdots &\\
             x_n& x_n &x_n&\ldots\\
             \vdots & \vdots& \vdots & \ddots
        \end{bmatrix} \odot_H  \begin{bmatrix}
            0&x_1&x_2&\ldots & x_n &\ldots\\
              0&x_1& x_2 &\ldots & x_n &\ldots\\
                0&x_1& x_2 &\ldots & x_n &\ldots\\
             \vdots&\vdots & \vdots & &\vdots &\ldots\\
             \vdots&\vdots & \vdots & & \vdots& \ddots
        \end{bmatrix}.
    \]
    Thus, one can easily see that $P_2(f)$ preserves the set of all decomposable elements and so it is rank one non-increasing. Also, $P_2(f^\dagger)$ is given by 
    \[
        P_2(f^\dagger)(x\otimes y)(i,j)=(f^\dagger(x)\otimes f^\dagger(y))(i,j)=x_{i+1}y_{j+1},
    \]
    i.e,
    \[
        x\otimes x\longmapsto \begin{bmatrix}
            x_2&x_2&x_2&\ldots\\
            \vdots & \vdots& \vdots &\\
             x_n& x_n &x_n&\ldots\\
             \vdots & \vdots& \vdots & \ddots
        \end{bmatrix} \odot_H  \begin{bmatrix}
            x_2&x_3&\ldots & x_n &\ldots\\
             \vdots &\vdots & & \vdots &\\
             x_2 &x_3&\ldots & x_n &\ldots\\
             \vdots & \vdots& &\vdots & \ddots
        \end{bmatrix}.
    \]
    Thus, $P_2(f^\dagger)$ also preserves decomposable elements, and hence it is rank one non-increasing.
\end{example}

\subsection{Drazin inverse}
Unlike the case of the Moore-Penrose inverse (Theorem~\ref{thm::rank-non-increasing-Moore-Penrose-inverse}), the Drazin inverse (if it exists) is always a preserver, as we show next.

\begin{theorem}\label{thm::inverse-of-rank-non-increasing}
    Let $X$ be a vector space and $X^{(2)}$ the second symmetric product space of $X$. Let $T\colon X^{(2)}\to X^{(2)}$ be a rank one non-increasing map. If $T^D$ exists, then $T^D$ is also rank one non-increasing.
\end{theorem}
\begin{proof}
    By Theorem~\ref{thm::Lim}, there are only two cases to consider. First, let $T$ be of form (ii) in Theorem~\ref{thm::Lim}. Then, for $k\in \Z_+$, we have
    $$T^{k+1}(x\otimes x)=(\phi(x\otimes x)(\phi(y\otimes y))^k) y\otimes y.$$ 
    Thus, $\text{Ind }T<\infty$ if and only if $\phi(y\otimes y)\neq 0$. Moreover, in that case, $\text{Ind }T=1$ and $T^D\colon X^{(2)}\to X^{(2)}$ is given by
    \[
        T^D(x\otimes x)=\frac{1}{\phi(y\otimes y)^2} \phi(x\otimes x) y\otimes y = \frac{1}{\phi(y\otimes y)^2} T(x\otimes x),\quad \forall x\in X.   
    \]
    Hence, $T^D$ is rank one non-increasing. Now, suppose $T$ is of form (i) in Theorem~\ref{thm::Lim}, i.e.
    \[
        T(x\otimes x)= c(f(x)\otimes f(x)), \quad \forall x\in X
    \]
    First we claim that the existence of $T^D$ implies the existence of $f^D$. Set $k:=\text{Ind }T$, and let $y\in \Range(f^k)$. Then $y=f^k(x)$ for some $x\in K$, and so $T^k(x\otimes x)=c^k (y\otimes y).$ Since $\Range(T^{k+1})=\Range(T^k)$, there exists $A\in X^{(2)}$ such that $T^{k+1}(A)= c^{k+1} (y\otimes y).$ Let $z_i\in X$ and $\lambda_i\in \R$ be such that $A=\sum_{i=1}^n \lambda_i (z_i\otimes z_i)$. Then 
    \[
        T^{k+1}(A)= c^{k+1} \left(\sum \lambda_i (f^{k+1}(z_i)\otimes f^{k+1}(z_i))\right).    
    \]
    Hence $y\otimes y=\sum \lambda_i (f^{k+1}(z_i)\otimes f^{k+1}(z_i))$, and so $y\in \text{span} \{ f^{k+1}(z_i)\colon 1\leq i\leq n\}$. Since $f$ is linear, it follows that $y\in \Range(f^{k+1})$, showing that $\Range (f^k)=\Range(f^{k+1}).$ Let $x\in \ker f^{k+1}$. Then $T^{k+1}(x\otimes x)=0$, and so $T^k(x\otimes x)=c ( f(x)\otimes f(x))=0$ showing that $x\in \ker f.$ Hence $f^D$ exists and $\text{Ind }f\leq k.$ We claim that 
    \[
        T^D = \frac{1}{c}~ P_2(f^D)=:S.
    \]
    Since $\{x\otimes x\mid x\in X\}$ spans $X^{(2)}$, it is enough to check the images on the decomposable elements. Let $x\in X$. Then
    \begin{align*}
        TS(x\otimes x)&=\frac{1}{c} T(f^D(x)\otimes f^D(x))\\
        &=ff^D(x)\otimes ff^D(x)\\
        &=f^Df(x)\otimes f^Df(x)\\
        &=ST(x\otimes x),
    \end{align*}
    \begin{align*}
        T^kST(x\otimes x)&=T^k(ff^D(x)\otimes ff^D(x))\\
        &=c^k (f^kf^Df(x)\otimes f^kf^Df(x))\\
        &=c^k(f^k(x)\otimes f^k(x))\\
        &=T^k(x\otimes x),
    \end{align*}
    \begin{align*}
        STS(x\otimes x)&=S(ff^D(x)\otimes ff^D(x))\\
        &=\frac{1}{c} (f^Dff^D(x)\otimes f^Dff^D(x))\\
        &=\frac{1}{c}(f^D(x)\otimes f^D(x))\\
        &=S(x\otimes x).
    \end{align*}
    This proves that $\text{Ind }T\leq k$ and $T^D=S$. Therefore, $T^D$ is rank one non-increasing by Theorem~\ref{thm::Lim}.
\end{proof}
\begin{remark}
    In the proof of Theorem~\ref{thm::inverse-of-rank-non-increasing}, we have shown that $\text{Ind }T=\text{Ind }f$. Hence, in particular, $T$ is bijective if and only if $f$ is bijective.
\end{remark}

We conclude with the following observation.

\begin{remark}
    Let $K:=\R^n_+$ in Corollary~\ref{Ex::COP}. Let $P\in M_n$ be a nonnegative invertible matrix such that $P^{-1}$ is not nonnegative. Then, the linear map $T\colon S^n\to S^n$, defined by $T(A)=PAP^\top$ for $A\in S^n$, preserves the set of all positive decomposable elements but $T^{-1}$, given by $T^{-1}(A)=P^{-1}A(P^{-1})^\top$, does not. For this reason, we do not pursue the problem of studying when a generalized inverse of a preserver of positive decomposable vectors, is again such a preserver.
\end{remark}


\end{document}